\documentclass[10pt, reqno]{amsart}
\usepackage[utf8]{inputenc}

\usepackage{graphicx}
\usepackage{comment}
\usepackage{fourier}
\usepackage[T1]{fontenc}
\usepackage[margin=1in]{geometry}
\usepackage{parskip}
\usepackage{amsthm}
\usepackage{amsmath}
\usepackage{amssymb}
\usepackage{mathtools}
\usepackage{apptools}
\usepackage{setspace}
\usepackage[
backend=biber,
style=numeric,
sorting=nyt,
maxnames=10,
giveninits=true
]{biblatex}
\newcommand {\R}{\mathbb{R}}

\newcommand {\Z}{\mathbb{Z}}
\newcommand {\N}{\mathbb{N}}

\newcommand {\C}{\mathbb{C}}

\newcommand{\la}{\lambda}
\newcommand{\re}{\operatorname{Re}}
\newcommand{\im}{\operatorname{Im}}
\newcommand{\sumast}{\mathop{\sum\nolimits^{\mathrlap{\ast}}}}

\newcommand{\Mod}[1]{\ (\mathrm{mod}\ #1)}
\newtheorem{thm}{Theorem}[section]

\newtheorem{lemma}[thm]{Lemma}

\newtheorem*{remark}{Remark}
\newtheoremstyle{named}{}{}{\itshape}{}{\bfseries}{.}{.5em}{\thmnote{#3}}
\theoremstyle{named}

\AtAppendix{\counterwithin{lemma}{section}}

\usepackage{multicol}
\usepackage{xcolor}
\usepackage{hyperref}
\hypersetup{
    colorlinks,
    linkcolor={red!50!black},
    citecolor={blue!50!black},
    urlcolor={blue!80!black}
}

\title{GL(2) Weyl Bound via a multiplicative character delta method}
\author{Wing Hong Leung}
\date{}

\begin{document}

\maketitle

\begin{abstract}
We use a trivial delta method with multiplicative characters for congruence detection to prove the Weyl bound for GL(2) in the $t$-aspect for a holomorphic or Hecke-Maass cusp form of arbitrary level and nebentypus. This parallels the work of Aggarwal \cite{taspecttrivial} in 2018, with the difference being the multiplicative character has a more natural connection to the twisted $L$-function. This provides another viewpoint to understand and explore the trivial and other delta methods.
\end{abstract}

\section{Introduction and Statement of Results}

Let $f$ be a fixed level $M$ Hecke cusp form of weight $k$ or Laplace eigenvalue $\frac{1}{4}+r^2$ with nebetypus $\psi$, and let $L(s,f)$ be its associated $L$-function. By the Phragmen-Lindel\"of principle, we have the convexity bound \begin{align*}
    L\left(\frac{1}{2}+it,f\right)\ll_{f,\epsilon} t^{\frac{1}{2}+\epsilon}
\end{align*}
for any $\epsilon>0$. Any bound of the form \begin{align*}
    L\left(\frac{1}{2}+it,f\right)\ll_{f,\epsilon} t^{\frac{1}{2}-\delta}
\end{align*}
for some $\delta>0$ is called a subconvexity bound, and the current best bound is the Weyl-type bound, saying \begin{align*}
    L\left(\frac{1}{2}+it,f\right)\ll_{f,\epsilon} t^{\frac{1}{3}+\epsilon}.
\end{align*}
This Weyl-type bound appears to be a natural barrier to many boundary problems.

This problem has been studied extensively and the Weyl bound has been established by many people. Good \cite{Good} was the first one to establish such a bound for the full modular group with tools from the spectral theory. Later Jutila \cite{jutilafirst} used Farey fractions and Voronoi summations to get the same result, which was then extended by Meurman \cite{meurman} to cover the Maass form case. Afterwards, Jutila \cite{jutilasecond} extended Good's proof to cover the Maass form case. Recently, many people have established results on this problem, with various constraints on the levels, see for example \cite{paperAc, paperAg, paperB}. As better machinery gets developed, these constraints, which are believed to be a mere annoyance, are gradually relaxed. In 2018 Aggarwal \cite{taspecttrivial} established the Weyl bound for holormorphic or Hecke-Maass cusp form of arbitrary level and nebentypus.

In \cite{taspecttrivial}, Aggarwal used the trivial delta method \cite[Lemma 1.1]{taspecttrivial} stated as follows. \begin{lemma}
    Let $V$ be a smooth real valued function, compactly supported inside $\R^+$ such that $V$ has bounded derivatives and $\int_0^\infty V(x)dx=1$. Let $X>1$ and $q\in\N$ be such that $q>X^{1+\epsilon}$ for some $\epsilon>0$. Then \begin{align*}
        \delta(n=0)=\frac{1}{q}\sum_{\alpha\Mod{q}}e\left(\frac{n\alpha}{q}\right)\int_0^\infty V(x)e\left(\frac{nx}{X}\right)dx+O_A\left(X^{-A}\right)
    \end{align*}
    for any $A>0$ and here $e(x)=e^{2\pi ix}$.
\end{lemma}
Other than this problem, this delta method and its variants have also been successfully applied to achieve results on related problems, see for example \cite{besseldelta, burgesstrivial}.

The idea behind the delta method is as follows. For $q>X^{1+\epsilon}$, \begin{align*}
    \delta(n=0)=\delta(q|n)\delta(|n|<X),
\end{align*}
and one uses additive characters to detect the divisibility condition $\delta(q|n)$ and the $x$-integral to detect $\delta(|n|<X)$. However one can also detect the divisibility condition by multiplicative characters, giving us the following variant of the trivial delta method.

\begin{lemma}\label{trivialdeltalemma}(Trivial Delta Method with Multiplicative Characters)
    Let $V\geq 0$ be a fixed smooth function compactly supported on $\R^+$. Let $n,r,p,N\in\N, K>1$ such that $n,r\sim N, pK>N^{1+\epsilon}$. We have for $p\nmid r$,
    \begin{align*}
        \delta(n=r)=&\frac{1}{\phi(p)}\sum_{\chi(p)}\chi(n)\overline{\chi}(r)\int_0^\infty V\left(\nu\right)\left(\frac{n}{r}\right)^{iK\nu}d\nu+O\left(N^{-A}\right)
    \end{align*}
    for any $A>0$.
\end{lemma}
\begin{proof}
    Repeated integration by parts on the $\nu$-integral gives us for any $j\in\N$, \begin{align*}
        \int_0^\infty V\left(\nu\right)\left(\frac{n}{r}\right)^{iK\nu}d\nu\ll_j \left(K\left(\log n-\log r\right)\right)^{-j}\ll_j \left(\frac{N}{K|n-r|}\right)^j.
    \end{align*}
    One saves an arbitrary amount unless $|n-r|\ll \frac{N^{1+\epsilon}}{K}$. The sum over $\chi$ mod $p$ detects $n\equiv r\Mod{p}$ as $p\nmid r$. The condition $pK>N^{1+\epsilon}$ guarantees $|n-r|> \frac{N^{1+\varepsilon}}{K}$ unless $n=r$.
\end{proof}

Using this delta method, we prove the Weyl bound for the GL(2) $L$-function in the $t$-aspect.

\begin{thm}\label{thm1}
    Let $f$ be a level $M$ Hecke cusp form of weight $k$ or Laplace eigenvalue $\frac{1}{4}+r^2$ with nebetypus $\psi$ and $L(s,f)$ be its associated $L$-function. Then \begin{align*}
        L\left(\frac{1}{2}+it,f\right)\ll_{f,\epsilon} t^{\frac{1}{3}+\epsilon}
    \end{align*}
    for any $\epsilon>0$.
\end{thm}

In the process, we prove the following theorem as an intermediate step.

\begin{thm}\label{thm2}
    Let $f$ be a level $M$ Hecke cusp form of weight $k$ or Laplace eigenvalue $\frac{1}{4}+r^2$ with nebetypus $\psi$ and $\la_f(n)$ be its Hecke eigenvalues. Let $V$ be a $t^\varepsilon$-inert function with a fixed compact support on $\R^+$. Then for $N<t^{1+\epsilon}$, \begin{align*}
        S(N):=\sum_n\la_f(n)n^{-it}V\left(\frac{n}{N}\right)\ll_{f,\epsilon} \min\left\{Nt^\epsilon,\sqrt{N}t^{\frac{1}{3}+\epsilon}\right\}.
    \end{align*}
\end{thm}

As mentioned above, this result has been proved in \cite{taspecttrivial} by Aggarwal and by various people in slightly more restrictive settings. While our result is not new and the approach parallels the ideas and steps in \cite{taspecttrivial}, other than having a small simplification in the integral analysis, the novelty here is the replacement of additive characters by multiplicative characters. Although various delta methods and circle methods have been successfully used to establish great results especially in recent years, few have tried to use multiplicative characters to detect congruences. Moreover, as the usage of delta methods gains increasing attention, many have questioned the underlying reasons why delta methods provide a means to solve boundary problems and how to understand the steps involved.

In this paper, we hope to provide another viewpoint as multiplicative characters have a more natural linkage with $L$-functions. In particular, when one applies the trivial delta with multiplicative characters on this GL(2) $t$-aspect problem, one encounters sums of the form \begin{align*}
    \sum_n\frac{\la_f(n)\chi(n)}{n^s} \quad \text{ and }\quad  \sum_r \frac{\chi(r)}{r^{s'}}
\end{align*}
with a sum of $\chi\Mod{p}$. These sums are equal to $L(s,f\otimes\chi)$ and $L(s',\chi)$ respectively when $\re(s),\re(s')>1$. One can then view these sums as $L$-functions and relate them as members of a family of $L$-functions as $\chi$ runs over the Dirichlet characters mod $p$. This is a small step towards understanding delta methods by a family approach, which is vastly used in the moment methods that were traditionally used to solve boundary problems before the delta methods saw more play.

Another difference in using multiplicative characters instead of additive characters in a congruence detection is that multiplicative characters allow us to use a functional equation of a certain $L$-function instead of Voronoi summation for dual summation. This completely removes the need to deal with Bessel transforms that appear in Voronoi summation, but one has to deal with the Gamma function instead. While Bessel functions are related to the Gamma function, and the analysis of either one is not particularly complicated in this fixed GL(2) problem, the absence of Bessel functions can make the analysis simpler in certain cases.

Finally, we would like to point out that detecting congruences with multiplicative characters is not limited only to the trivial delta method shown in this paper, but may apply to any delta method that uses arithmetic congruences. While we don't anticipate improvements in results established by using additive characters for congruences, this may simplify the analysis in certain cases and hopefully provide a better understanding of how delta methods work as explained above.

\subsection*{Acknowledgments}

The author thanks R. Holowinsky for suggesting the problem and for his guidance, and K. Aggarwal for extensive discussions, valuable comments and his careful checking of the computations in this paper.

\section{Notations}

Let $\epsilon>0$, $a, b\in \C$, $c\in\Z$ and $t\rightarrow\infty$.

\begin{spacing}{1.5} 
 \begin{tabbing} 
 \= Simbolo \= Separador \= Significado \+ \kill 
 {\bf Symbol} \> \> {\bf Meaning} \\ \\
 $a\ll b$ \> \> There exists a constant $C>0$ such that $|a|\leq C|b|$\\
 $a\sim b$ \> \> $b\ll a\ll b$\\
 $a\sim_\epsilon b$ \> \> $bt^\epsilon\ll a\ll bt^\epsilon$\\
 $(a,b)$ \> \> The greatest common divisor of $a$ and $b$\\
 $e(x)$ \> \> $e^{2\pi ix}$\\
 $\displaystyle\sumast_{\alpha\Mod{c}}$ \> \> Sum over $\alpha\in (\Z/c)^*$\\
 $\displaystyle\sumast_{\chi\Mod{p}}$ \> \> Sum over primitive characters mod $p$\\
 $S(a,b;c)$ \> \> The Kloosterman sum $\displaystyle\sumast_{\alpha\Mod{c}}e\left(\frac{a\alpha+b\overline{\alpha}}{c}\right)$
 \end{tabbing}
\end{spacing}

We say $f$ is $X$-inert if it is smooth and for any $x\in\R$, $j\geq 0$, it satisfies \begin{align*}
    x^jf^{(j)}(x)\ll_{j,\epsilon} X^jt^\epsilon.
\end{align*}
Note that a similar definition is introduced by \cite{inert}. We include $t^\epsilon$ for convenience.

\section{Proof Sketch}

As $f$ is fixed and $t\rightarrow\infty$, we are going to assume $M=1$ in the sketch. We first apply an approximate functional equation to reduce the proof of Theorem \ref{thm1} to proving Theorem \ref{thm2}. Then we apply the trivial delta method in Lemma \ref{trivialdeltalemma} with an average over primes $p\sim P$ to separate the variables $n$ and $r$, giving us roughly \begin{align*}
    S(N)\sim \frac{1}{P}\sum_{p\sim P}\frac{1}{p}\sum_{\chi(p)}\int_0^\infty V\left(\nu\right)\sum_{n\sim N}\la_f(n)\chi(n)n^{iK\nu}\sum_{r\sim N}\overline{\chi}(r)r^{-i(t+K\nu)}d\nu,
\end{align*}
for $PK>N$. This is done in Section \ref{SectSetup}. The bound at this stage is $N^2$.

Next we perform, in Section \ref{SectDualSum}, dual summations to the $n$-sum and $r$-sum by doing Mellin inversion and using the functional equations of the $L$-functions $L(s,\chi)$ and $L(s,f\otimes\chi)$ respectively. The new length of $n$ becomes $\frac{P^2K^2}{N}$ and that of $r$ becomes $\frac{Pt}{N}$ as we take $K<t$. This implies that we save $\frac{N}{PK}$ and $\frac{N}{\sqrt{Pt}}$ in the $n$ and $r$ sum respectively. Summing over $\chi\Mod{p}$ and cleaning up the $\nu$-integral by stationary phase analysis gives us $\sqrt{P}$ and $\sqrt{K}$ savings respectively, and we end up with roughly \begin{align*}
    S(N)\sim\frac{N^2}{P^3K^\frac{3}{2}\sqrt{t}}\sum_{p\sim P}p^{-it}\sum_{n\sim\frac{P^2K^2}{N}}\overline{\la_f(n)}\sum_{r\sim \frac{Pt}{N}}r^{it}e\left(-\frac{n\overline{r}}{p}-\sqrt{\frac{2nt}{\pi pr}}-\frac{n}{2pr}\right).
\end{align*}
The bound at this stage is $P\sqrt{Kt}$, which can also be seen by the savings from $N^2$ from the process above. This is not good enough as we have the constraint $PK>N$. We have to save a bit more than $\sqrt{P}$ to beat convexity.

To break the structure, we perform the Cauchy-Schwarz inequality in Section \ref{SectCS} to eliminiate the GL(2) coefficients $\la_f(n)$ by taking out the $n$-sum. This gives us roughly \begin{align*}
    S(N)\leq\frac{N^\frac{3}{2}}{P^2\sqrt{Kt}}\left(\sum_{n\sim\frac{P^2K^2}{N}}\left|\sum_{p\sim P}p^{-it}\sum_{r\sim\frac{Pt}{N}}r^{it}e\left(\frac{n\overline{r}}{p}-\sqrt{\frac{2nt}{\pi pr}}-\frac{n}{2pr}\right)\right|^2\right)^\frac{1}{2}.
\end{align*}
Opening up the square, we perform Poisson summation on the $n$-sum, giving us roughly \begin{align*}
    S(N)\leq \frac{N\sqrt{K}}{P\sqrt{t}}\left(\sum_{|n|\ll\frac{N}{K}}\sum_{p_1\sim P}\sum_{p_2\sim P}\left(\frac{p_2}{p_1}\right)^{it}\sum_{r_1\sim\frac{Pt}{N}}\sum_{r_2\sim\frac{Pt}{N}}\left(\frac{r_1}{r_2}\right)^{it}\delta\left(n\equiv \overline{r_2}p_1-\overline{r_1}p_2\Mod{p_1p_2}\right)\times \text{ some integral}\right)^\frac{1}{2}.
\end{align*}
We then split the treatment into whether $p_1=p_2$ or not. For the case $p_1=p_2$, we subdivide further into the diagonal case with $r_1=r_2$ and the case $r_1\neq r_2$. The diagonal case is bounded by $\sqrt{NK}$ as we get square root savings in the $p$ and $r$-sums, and the $r_1\neq r_2$ case is bounded by $K^\frac{1}{4}\sqrt{t}$ as we get a $\sqrt{P}$ savings from the $p$-sum, a $\sqrt{P}$ savings by the congruence and a $K^\frac{1}{4}$ savings from the integral. For the off-diagonal $p_1\neq p_2$ case, we lose $\sqrt{\frac{N}{K}}$ on the $n$-sum length, but save $\sqrt{P}$ by the congruence on each $r_1$ and $r_2$ and also save $K^\frac{1}{4}$ in the integral, giving us the bound $\sqrt{Nt}K^{-\frac{1}{4}}$. The optimal choice of $K=t^\frac{2}{3}$ gives the desired bound $S(N)\ll \sqrt{N}t^{\frac{1}{3}+\epsilon}$.

\section{Preliminaries and Lemmas}

Here we collect the facts we need about Hecke cusp forms, $L$-functions and the standard tools we need for this paper.

\subsection{Hecke cusp forms}

For a Hecke cusp form $f$, the $L$-function associated to $f$ is defined by the analytic continuation of the series given by \begin{align*}
    L(s,f)=\prod_p\left(1-\frac{\alpha_{f,1}(p)}{p^s}\right)^{-1}\left(1-\frac{\alpha_{f,2}(p)}{p^s}\right)^{-1}=\sum_{n\geq1}\la_f(n)n^{-s} \text{  for } \re(s)>1.
\end{align*}
When $f$ is primitive, $\la_f(n)$ is equal to the Hecke eigenvalues of $f$.

If $f$ is holomorphic, $\la_f(n)$ satisfies the following bound by Deligne, \begin{align}\label{DeligneBound}
    |\la_f(n)|\leq \tau(n)\ll n^\epsilon
\end{align} 
for any positive integer $n$. If $f$ is a Maass cusp form, a proof for a bound similar to (\ref{DeligneBound}) has not been established. However, $\la_f(n)$ do satisfy "Ramanujan-Petersson on average" by Rankin-Selberg theory, that is \begin{align}\label{RPavg}
    \sum_{1\leq n\leq x}\left|\la_f(n)\right|^2\ll_{f,\epsilon}x^{1+\epsilon}.
\end{align}

\subsection{Functional Equations}

Let $\chi$ be a primitive Dirichlet character mod $p$. Let $f$ be either a level $M$ holomorphic cusp form of weight $k$ with nebentypus $\psi$ or a level $M$ Hecke-Maass cusp form with Laplace eigenvalue $\frac{1}{4}+r^2$ of nebentypus $\psi$. Let \begin{align*}
    a=\begin{cases}0&\text{ if }\chi(-1)=1\\
1&\text{ if }\chi(-1)=-1,\end{cases}
\end{align*}
\begin{align*}
    \delta_f=\begin{cases}0&\text{ if } f \text{ is holomorphic or an even Maass form} \\
    1&\text{ if } f \text{ is an odd Maass form}\end{cases}
\end{align*}
and let \begin{align*}
    \gamma_f(s)=\begin{cases}(2\pi)^{-s}\Gamma\left(s+\frac{k-1}{2}\right) & \text{ if } f \text{ is holomorphic}\\
    \pi^{-s}\Gamma\left(\frac{s+\delta_f+ir}{2}\right)\Gamma\left(\frac{s+\delta_f-ir}{2}\right)& \text{ if } f \text{ is Maass.}\end{cases}
\end{align*}
Define \begin{align*}
    \Lambda(s,\chi)=&\left(\frac{\pi}{p}\right)^{-\frac{s+a}{2}}\Gamma\left(\frac{s+a}{2}\right)L(s,\chi),\\
    \Lambda(s,f)=&M^\frac{s}{2} \gamma_f(s) L(s, f)\\
    \Lambda(s,f\otimes\chi)=&(\sqrt{M}p)^s\gamma_f(s+a(1-2\delta_f))L(s,f\otimes\chi),
\end{align*}
and let $\overline{f}$ be the dual cusp form defined by $\la_{\overline{f}}(n)=\overline{\la_f(n)}$. Then for some root number $\epsilon(f)\in\C$ satisfying $|\epsilon(f)|=1$, we have

\begin{lemma}\label{FElemma}(Functional Equation)
\begin{align*}
    \Lambda(s,\chi)=&i^{-a}\epsilon_\chi\Lambda(1-s,\overline{\chi})\\
    \Lambda(s,f)=&\epsilon(f)\Lambda(1-s,\overline{f})\\
    \Lambda(s,f\otimes\chi)=&\epsilon(f)\psi(p)\chi(M)\epsilon_\chi^2\Lambda(1-s,\overline{f}\otimes\overline{\chi}).
\end{align*}
\end{lemma}
For a proof, see for example \cite{secondmoment, kmv, winnieli}.

\subsection{Approximate Functional Equation}

In order to express $L\left(\frac{1}{2}+it,f\right)$ in a more tractable way, we apply the following approximate functional equation. See for example \cite[Thm 5.3, Prop. 5.4]{IwaniecKowalski} for a proof.

\begin{lemma}\label{AFElemma}
    Let $G(u)$ be an even holormophic function bounded in the strip $-4<\re(u)<4$ and normalized by $G(0)=1$. Then for $s$ in the strip $0\leq \re(s)\leq 1$, we have \begin{align*}
        L(s,f)=\sum_n\frac{\la_f(n)}{n^s}V_s\left(\frac{n}{\sqrt{M}}\right)+\epsilon(f)M^{\frac{1}{2}-s}\frac{\gamma_f(1-s)}{\gamma_f(s)}\sum_n\frac{\la_f(n)}{n^{1-s}}V_{1-s}\left(\frac{n}{\sqrt{M}}\right)+R,
    \end{align*}
    where $R\ll1$ and $V_s(y)$ satisfies for any $j\geq0, A>0$, \begin{align*}
        y^jV_s^{(j)}(y)\ll_{j,A}\left(1+\frac{y}{\sqrt{t}}\right)^{-A}.
    \end{align*}
\end{lemma}

\subsection{Gamma Function}

To analyse the Gamma function, we have the following Stirling's approximation. \begin{lemma}
    Let $a_0=1$, $a_1=-\frac{1}{12}$ and let $z\rightarrow\infty$ with $|\arg(z)|\leq\pi-\delta$ for any $\delta>0$. There exists $a_n\in\C$ for any $n\geq 2$ such that for any $N\geq1$, \begin{align*}
        \Gamma(z)=\sqrt{2\pi}\frac{z^{z-\frac{1}{2}}}{e^z}\left(\sum_{n=0}^{N-1}\frac{a_n}{z^n}+O\left(|z|^{-N}\right)\right).
    \end{align*}
\end{lemma}

By Stirling's approximation, we have the following two lemmas.

\begin{lemma}\label{GammaRatioBound}
    Let $\alpha,\beta,\tau\in\R$ such that $\alpha,\beta$ is fixed. Then \begin{align*}
        \left|\frac{\Gamma(\alpha+i\tau)}{\Gamma(\beta-i\tau)}\right|\ll (1+|\tau|)^{\alpha-\beta}.
    \end{align*}
\end{lemma}
\begin{proof}
    When $|\tau|\ll1$, this is trivial. So we are left with the case $|\tau|\rightarrow\infty$. By Stirling's approximation, we have \begin{align*}
        \left|\frac{\Gamma(\alpha+i\tau)}{\Gamma(\beta-i\tau)}\right|\ll\frac{(\alpha^2+\tau^2)^{\frac{\alpha}{2}-\frac{1}{4}}}{(\beta^2+\tau^2)^{\frac{\beta}{2}-\frac{1}{4}}}e^{-\tau\arctan\frac{\tau}{\alpha}+\tau\arctan\frac{\tau}{\beta}-\alpha+\beta}\ll|\tau|^{\alpha-\beta}.
    \end{align*}
    To get the second inequality we used the Taylor expansion $\arctan(x)=\frac{\pi}{2}-\arctan(x^{-1})=\frac{\pi}{2}-x^{-1}+O(x^{-3})$.
\end{proof}

\begin{lemma}\label{GammaRatioLemma}
Let $\alpha, \beta, \tau\in\R$ such that $\alpha$ is fixed, $\tau\ll t^\epsilon\ll \beta t^{-\epsilon}$. Then \begin{align*}
    \frac{\Gamma(\alpha-i(\beta+\tau))}{\Gamma(\alpha+i(\beta+\tau))}=\left(\frac{|\beta|}{e}\right)^{-2i\beta}|\beta|^{-2i\tau}e\left(-\frac{\tau^2}{2\pi\beta}\right)W_\alpha(\beta)+O(\beta^{-2}),
\end{align*}
where \begin{align*}
    W_\alpha(\beta)=e\left(\frac{1-2\alpha}{2\pi}\left(\frac{\pi}{2}-\frac{\alpha}{\beta}\right)\right)\left(1-\frac{i\alpha^2}{\beta}\right)\frac{1+\frac{1}{12(\alpha-i\beta)}}{1+\frac{1}{12(\alpha+i\beta)}}
\end{align*}
is $1$-inert in $\beta$.
\end{lemma}
\begin{proof}
By Stirling's Approximation, we have \begin{align}\label{GammaRatioTemp}
    \frac{\Gamma(\alpha-i(\beta+\tau))}{\Gamma(\alpha+i(\beta+\tau))}=\left(\frac{\alpha^2+(\beta+\tau)^2}{e^2}\right)^{-i(\beta+\tau)}e^{(2\alpha-1)i\arg(\alpha-i(\beta+\tau))}\left(\frac{1+\frac{1}{12(\alpha-i(\beta+\tau))}}{1+\frac{1}{12(\alpha+i(\beta+\tau))}}+O\left(\beta^{-2}\right)\right).
\end{align}
By Taylor expansions, we get \begin{align*}
    \frac{1+\frac{1}{12(\alpha-i(\beta+\tau))}}{1+\frac{1}{12(\alpha+i(\beta+\tau))}}=\frac{1+\frac{1}{12(\alpha-i\beta)}}{1+\frac{1}{12(\alpha+i\beta)}}+O(\beta^{-2}),
\end{align*}
\begin{align*}
    e\left(\frac{2\alpha-1}{2\pi}\arg\left(\alpha-i(\beta+\tau)\right)\right)=&e\left(-\frac{2\alpha-1}{2\pi}\arctan\left(\frac{\beta+\tau}{\alpha}\right)\right)\\
    =&e\left(\frac{1-2\alpha}{2\pi}\left(\frac{\pi}{2}-\arctan\left(\frac{\alpha}{\beta+\tau}\right)\right)\right)=e\left(\frac{1-2\alpha}{2\pi}\left(\frac{\pi}{2}-\frac{\alpha}{\beta}\right)\right)+O\left(\beta^{-2}\right)
\end{align*}
and \begin{align*}
    \left(\frac{\alpha^2+(\beta+\tau)^2}{e^2}\right)^{-i(\beta+\tau)}=&\left(\frac{|\beta|}{e}\right)^{-2i(\beta+\tau)}\left(1+\frac{\tau}{\beta}\right)^{-2i(\beta+\tau)}\left(1+\left(\frac{\alpha}{\beta}\right)^2\right)^{-i(\beta+\tau)}\\
    =&\left(\frac{|\beta|}{e}\right)^{-2i(\beta+\tau)}e\left(-\frac{\beta+\tau}{\pi}\log\left(1+\frac{\tau}{\beta}\right)\right)\left(1-i(\beta+\tau)\left(\frac{\alpha}{\beta}\right)^2+O(\beta^{-2})\right)\\
    =&\left(\frac{|\beta|}{e}\right)^{-2i(\beta+\tau)}e\left(-\frac{\beta+\tau}{\pi}\left(\frac{\tau}{\beta}-\frac{\tau^2}{2\beta^2}+O(\beta^{-3})\right)\right)\left(1-\frac{i\alpha^2}{\beta}\right)+O(\beta^{-2})\\
    =&\left(\frac{|\beta|}{e}\right)^{-2i\beta}|\beta|^{-2i\tau}e\left(-\frac{\tau^2}{2\pi\beta}\right)\left(1-\frac{i\alpha^2}{\beta}\right)+O(\beta^{-2}).
\end{align*}
Putting the above into (\ref{GammaRatioTemp}), we get \begin{align*}
    \frac{\Gamma(\alpha-i(\beta+\tau))}{\Gamma(\alpha+i(\beta+\tau))}=\left(\frac{|\beta|}{e}\right)^{-2i\beta}|\beta|^{-2i\tau}e\left(-\frac{\tau^2}{2\pi\beta}+\frac{1-2\alpha}{2\pi}\left(\frac{\pi}{2}-\frac{\alpha}{\beta}\right)\right)\left(1-\frac{i\alpha^2}{\beta}\right)\frac{1+\frac{1}{12(\alpha-i\beta)}}{1+\frac{1}{12(\alpha+i\beta)}}+O\left(\beta^{-2}\right).
\end{align*}
\end{proof}

\begin{remark}
    We only wrote out the first two terms in the asymptotic expansion of $\,\,\Gamma$ to get the above lemma. If we need a better error term we can always write out more terms in the asymptotic expansion.
\end{remark}

\subsection{Mellin transform}

Let $V$ be a $t^\epsilon$-inert function compactly supported on $(t^{-\epsilon},t^\epsilon)$ and let $$\Tilde{V}(s)=\int_0^\infty V(x)x^{s-1}dx$$ be the Mellin transform of $V$. Then repeated integration by parts on the $x$-integral gives the following lemma. \begin{lemma}\label{MellinBound}
    Let $s=\sigma+i\tau$. For any $j\geq 0$, we have \begin{align*}
        \Tilde{V}(s)\ll_{\sigma,j} \left(\frac{t^\epsilon}{1+|\tau|}\right)^j.
    \end{align*}
\end{lemma}

\subsection{Oscillatory Integrals}

The first tool to study oscillatory integrals is the stationary phase analysis shown in the following lemma, which is Lemma 3.1 in \cite{inert}.

\begin{lemma}\label{stat}
    Suppose that $w$ is an $X$-inert function compactly supported on $[Z,2Z]$, so that $w^{(j)}(t)\ll \left(Z/X\right)^{-j}$. Also, suppose that $\phi$ is smooth and satisfies $\phi^{(j)}(t)\ll\frac{Y}{Z^j}$ for some $\frac{Y}{X^2}\geq R\geq1$ and all $t$ in the support of $w$. Let \begin{align*}
        I=\int_{-\infty}^\infty w(t)e^{i\phi(t)}dt.
    \end{align*}
    \begin{enumerate}
        \item If $|\phi'(t)|\gg\frac{Y}{Z}$ for all $t$ in the support of $w$, then $I\ll_A ZR^{-A}$ for any $A>0$.
        \item If $\phi''(t)\gg\frac{Y}{Z^2}$ for all $t$ in the support of $w$, and there exists $t_0\in\R$ such that $\phi'(t_0)=0$, then \begin{align*}
            I=\frac{e^{i \phi\left(t_{0}\right)}}{\sqrt{\phi^{\prime \prime}\left(t_{0}\right)}} F\left(t_{0}\right)+O_{A}\left(Z R^{-A}\right)
        \end{align*}
        where $F$ is a $X$-inert function (depending on $A$) supported on $t_0\sim Z$.
    \end{enumerate}
\end{lemma}

\begin{remark}
    ~\begin{itemize}
        \item By applying a smooth dyadic subdivision, the condition of $w$ being compactly supported on $[Z,2Z]$ for Lemma \ref{stat} can be relaxed to $w$ being compactly supported on $(t^{-\epsilon},t^\epsilon)$. We will apply Lemma \ref{stat} with $w$ being a $t^\epsilon$-inert function supported on $(t^{-\epsilon}, t^\epsilon)$.
        \item As described in the notation section, there is a small difference of $t^\epsilon$ in our definition of inert and that of \cite{inert}, but this does not affect the statement or proof of the lemma.
    \end{itemize}
\end{remark}

We also have the following standard simple second derivative bound that requires fewer conditions, see for example \cite[Lemma 5]{2ndDerBound}.

\begin{lemma}\label{2ndDBound}
    Let $V$ be a $t^\epsilon$-inert function compactly supported in $(t^{-\epsilon},t^\epsilon)$ and let $f$ be a real smooth function. Let $r>0$ such that $|f''(x)|>r$ for any $x\in\mathrm{supp}(V)$, then \begin{align*}
        \int_0^\infty V(x)e(f(x))dx\ll \frac{t^\epsilon}{\sqrt{r}}.
    \end{align*}
\end{lemma}

\section{Setup}\label{SectSetup}

The calculations in \cite{taspecttrivial} are focused on Hecke Maass cusp forms.
For the ease of demonstration in this paper, our calculations are primarily focused on holomorphic cusp forms, but we do point out the essential adjustments needed for Hecke Maass cusp forms.

Let $f$ be a fixed level $M$ holomorphic cusp form of weight $k$ with nebentypus $\psi$. By the approximate functional equation in Lemma \ref{AFElemma}, the bound for the Gamma function in Lemma \ref{GammaRatioBound} and a smooth dyadic subdivision, we have \begin{align}\label{ApplyAFE}
    L\left(\frac{1}{2}+it,f\right)\ll t^\epsilon\sup_{1\leq N\ll t^{1+\epsilon}}\frac{S(N)}{\sqrt{N}},
\end{align}
where \begin{align}
    S(N):=\sum_n\la_f(n)n^{-it}V\left(\frac{n}{N}\right)
\end{align}
for some $t^\varepsilon$-inert function $V$ compactly supported on $(1,2)$. This shows that Theorem \ref{thm2} implies Theorem \ref{thm1}.

Note that by (\ref{DeligneBound}) or (\ref{RPavg}) for the Maass form case, we have \begin{align}\label{trivialBound}
    S(N)\ll Nt^\epsilon.
\end{align}

Let $U$ be a fixed smooth function compactly supported on $\left(\frac{1}{2},\frac{5}{2}\right)$ and $U(x)=1$ for $x\in [1,2]$. Then \begin{align*}
    S(N)=\sum_n\la_f(n)V\left(\frac{n}{N}\right)\sum_r r^{-it}U\left(\frac{r}{N}\right)\delta(n=r).
\end{align*}
Now we apply Lemma \ref{trivialdeltalemma} to separate the oscillation in $S(N)$. Let $P,K\geq t^\epsilon$ be parameters, and $\mathcal{P}$ be the set of primes in $[P,2P]$. Let $P^*=|\mathcal{P}|$, and note that $P^*>Pt^{-\epsilon}$. Applying the above delta method together with an average of $p\in\mathcal{P}$, we get \begin{align}\label{FirstSplit}
    S(N)=S^*(N)+S_0(N)+S_1(N),
\end{align}
where \begin{align}\label{S0Bound}
    S_0(N)=\frac{1}{P^*}\sum_{p\in\mathcal{P}}\sum_{p|n}\la_f(n)V\left(\frac{n}{N}\right)\ll \frac{Nt^\epsilon}{P}
\end{align}
by (\ref{DeligneBound}) or (\ref{RPavg}) for the Maass form case,
\begin{align}
    S_1(N)=\frac{1}{P^*}\sum_{p\in \mathcal{P}}\frac{1}{\phi(p)}\int_0^\infty V\left(\nu\right)\sum_{p\nmid n}\la_f(n)n^{iK\nu}V\left(\frac{n}{N}\right)\sum_{p\nmid r}r^{-i(t+K\nu)}U\left(\frac{r}{N}\right)d\nu
\end{align}
and \begin{align}
    S^*(N)=\frac{1}{P^*}\sum_{p\in \mathcal{P}}\frac{1}{\phi(p)}\sumast_{\chi(p)}\int_0^\infty V\left(\nu\right)\sum_{n}\la_f(n)\chi(n)n^{iK\nu}V\left(\frac{n}{N}\right)\sum_{r}\overline{\chi}(r)r^{-i(t+K\nu)}U\left(\frac{r}{N}\right)d\nu.
\end{align}

We first bound $S_1(N)$. By repeated integration by parts on the $\nu$-integral as shown in the proof of Lemma \ref{trivialdeltalemma}, one saves an arbitrary amount unless $|n-r|<\frac{N}{K}t^\epsilon$. Bounding $\la_f(n)$ with (\ref{DeligneBound}) or (\ref{RPavg}) for the Maass form case, we get for $K<Nt^{-\epsilon}$, \begin{align}
    S_1(N)\ll\frac{N^2t^\epsilon}{PK}.
\end{align}
Hence for $K<Nt^{-\epsilon}$, we have \begin{align}\label{SNS*Nrelation}
    S(N)=S^*(N)+O\left(\frac{Nt^\epsilon}{P}+\frac{N^2t^\epsilon}{PK}\right).
\end{align}

\section{Dual summations and Clean Up}\label{SectDualSum}

Now we are left to bound \begin{align}
    S^*(N)=\frac{1}{P^*}\sum_{p\in \mathcal{P}}\frac{1}{\phi(p)}\sumast_{\chi(p)}\int_0^\infty V\left(\nu\right)\sum_{n}\la_f(n)\chi(n)n^{iK\nu}V\left(\frac{n}{N}\right)\sum_{r}\overline{\chi}(r)r^{-i(t+K\nu)}U\left(\frac{r}{N}\right)d\nu.
\end{align}
We perform dual summations by using functional equations of various $L$-functions.

\subsection{\texorpdfstring{$r$}{r}-sum functional equation}

Let $$\Tilde{U}(s)=\int_0^\infty U(x)x^{s-1}dx$$ be the Mellin transform of $U$. By Mellin inversion and the $L$-function representation, we have for any $\sigma>0$,  \begin{align}
    \sum_r\overline{\chi}(r)r^{-i(t+K\nu)}U\left(\frac{r}{N}\right)&=\frac{1}{2\pi i}\int_{(\sigma)}\Tilde{U}(s)N^s\sum_r\overline{\chi}(r)r^{-i(t+K\nu)-s}ds\nonumber\\
    &=\frac{1}{2\pi i}\int_{(\sigma)}\Tilde{U}(s)N^sL(s+i(t+K\nu),\overline{\chi})ds.
\end{align}
Applying the functional equation of the Dirichlet $L$-function in Lemma \ref{FElemma}, the $r$-sum becomes
\begin{align*}
    &\frac{i^{-a}\epsilon_{\overline{\chi}}}{2\pi i}\int_{(\sigma)}\Tilde{U}(s)N^s\left(\frac{\pi}{p}\right)^{s-\frac{1}{2}+i(t+K\nu)}\gamma(s+i(t+K\nu),a)L(1-s-i(t+K\nu),\chi)ds,
\end{align*}
where \begin{align*}
    \gamma(s,a)=\frac{\Gamma\left(\frac{1-s+a}{2}\right)}{\Gamma\left(\frac{s+a}{2}\right)}.
\end{align*}
By Lemma \ref{MellinBound} on $\Tilde{V}$ and Lemma \ref{GammaRatioBound} giving $\left|\gamma(\sigma+i(t+K\nu+\tau),a)\right| \ll (1+|t+K\nu+\tau|)^{-\sigma}$, one saves an arbitrary amount unless $|\im(s)|\ll t^\epsilon$. Hence we can shift the contour to $-M$ for some $M>0$, without hitting any poles since $\chi$ is primitive, to get \begin{align}
    \frac{i^{-a}\epsilon_{\overline{\chi}}}{2\pi i}\left(\frac{\pi}{p}\right)^{-\frac{1}{2}+i(t+K\nu)}\sum_r\chi(r)r^{-1+i(t+K\nu)}\int_{\substack{(-M)\\|\im(s)|\ll t^\epsilon}}\Tilde{U}(s)\left(\frac{\pi Nr}{p}\right)^s\gamma(s+i(t+K\nu),a)ds.
\end{align}
Again bounding $\left|\gamma(\sigma+i(t+K\nu+\tau),a)\right| \ll (1+|t+K\nu+\tau|)^{\sigma}$, one saves an arbitrary amount unless $r\ll \frac{Pt^{1+\epsilon}}{N}$ as we have chosen $K<Nt^{-10\epsilon}<t^{1-\epsilon}$. Similarly, by shifting the contour to the right without hitting any poles since $|\im(s)|\ll t^\epsilon$, one saves an arbitrary amount unless $r\gg\frac{Pt^{1-\epsilon}}{N}$. Restricting the length of $$r\sim_\epsilon\frac{Pt}{N},$$
we shift the contour back to $\frac{1}{2}$, giving us \begin{align}\label{rsumafterfunctional}
    \frac{i^{-a}\epsilon_{\overline{\chi}}}{2\pi i}\left(\frac{\pi}{p}\right)^{-\frac{1}{2}+i(t+K\nu)}\sum_r\chi(r)r^{-1+i(t+K\nu)}V_\epsilon\left(\frac{rN}{Pt}\right)\int_{\substack{\left(\frac{1}{2}\right)\\|\im(s)|\ll t^\epsilon}}\Tilde{U}(s)\left(\frac{\pi Nr}{p}\right)^s\gamma(s+i(t+K\nu),a)ds+O\left(t^{-2020}\right),
\end{align}
where $V_\epsilon$ is a $t^\epsilon$-inert function compactly supported on $(t^{-2\epsilon},t^{2\epsilon})$ and is equal to 1 on $[t^{-\epsilon},t^\epsilon]$.

Applying Lemma \ref{GammaRatioLemma} for $s=\frac{1}{2}+i\tau$ with $|\tau|\ll t^\epsilon$, we have \begin{align*}
    \gamma(s+i(t+K\nu),a)=&\frac{\Gamma\left(\frac{\frac{1}{2}-i(t+K\nu+\tau)+a}{2}\right)}{\Gamma\left(\frac{\frac{1}{2}+i(t+K\nu+\tau)+a}{2}\right)}\\
    =&\left(\frac{t+K\nu}{2e}\right)^{-i(t+K\nu)}\left(\frac{t+K\nu}{2}\right)^{-i\tau}e\left(-\frac{\tau^2}{4\pi(t+K\nu)}\right)W_{\frac{1}{4}+\frac{a}{2}}\left(\frac{t+K\nu}{2}\right)+O(t^{-2}),
\end{align*}
where $W_{\frac{1}{4}+\frac{a}{2}}$ is defined in Lemma \ref{GammaRatioLemma}. By the same lemma, $W_{\frac{1}{4}+\frac{a}{2}}$ is $1$-inert. Hence we have \begin{align}
    &\int_{\substack{\left(\frac{1}{2}\right)\\ |\im(s)|\ll t^\epsilon}}\Tilde{U}(s)\left(\frac{\pi Nr}{p}\right)^s\gamma(s+i(t+K\nu),a)ds\nonumber\\
    =&W_{\frac{1}{4}+\frac{a}{2}}\left(\frac{t+K\nu}{2}\right)\left(\frac{t+K\nu}{2e}\right)^{-i(t+K\nu)}\int_{\substack{\left(\frac{1}{2}\right)\\ |\im(s)|\ll t^\epsilon}}\Tilde{U}(s)\left(\frac{\pi Nr}{p}\right)^s\left(\frac{t+K\nu}{2}\right)^{-i\tau}e\left(-\frac{\tau^2}{4\pi(t+K\nu)}\right)\left(1+O\left(t^{-2}\right)\right)ds\nonumber\\
    =&i\sqrt{\frac{\pi Nr}{p}}W_{\frac{1}{4}+\frac{a}{2}}\left(\frac{t+K\nu}{2}\right)\left(\frac{t+K\nu}{2e}\right)^{-i(t+K\nu)}\int_{|\tau|\ll t^\epsilon}\Tilde{U}\left(\frac{1}{2}+i\tau\right)\left(\frac{2\pi tr}{p(t+K\nu)}\right)^{i\tau}e\left(-\frac{\tau^2}{4\pi(t+K\nu)}\right)d\tau+O\left(\frac{t^\epsilon}{\sqrt{N}t}\right).
\end{align}

Putting this back into the $r$-sum after shifting the contour to $\frac{1}{2}$ in (\ref{rsumafterfunctional}), and bounding the error term trivially with absolute value, we get \begin{align}\label{rsumafterdual}
    \sqrt{N}\epsilon_{\overline{\chi}}\sum_r\frac{\chi(r)}{\sqrt{r}}\left(\frac{p(t+K\nu)}{2\pi er}\right)^{-i(t+K\nu)}V_\epsilon\left(\frac{rN}{Pt}\right)W_{1,p,a}(r,\nu)+O\left(\frac{\sqrt{p}}{\sqrt{N}t^{1-\epsilon}}\right),
\end{align}
where \begin{align*}
    W_{1,p,a}(r,\nu)=\frac{i^{-a}}{2\pi}W_{\frac{1}{4}+\frac{a}{2}}\left(\frac{t+K\nu}{2}\right)\int_{|\tau|< t^\epsilon}\Tilde{U}\left(\frac{1}{2}+i\tau\right)\left(\frac{2\pi tr}{p(t+K\nu)}\right)^{i\tau}e\left(-\frac{\tau^2}{4\pi(t+K\nu)}\right)d\tau
\end{align*}
is $t^\epsilon$-inert in both $r$ and $\nu$.

Putting the $r$-sum back into $S^*(N)$ and bounding the error term with (\ref{DeligneBound}) or (\ref{RPavg}) for the Maass form case, we get \begin{align}\label{S*afterrdual}
    S^*(N)=&\frac{\sqrt{N}}{P^*}\sum_{p\in \mathcal{P}}\frac{1}{\phi(p)}\sumast_{\chi(p)}\epsilon_{\overline{\chi}}\int_0^\infty V\left(\nu\right)\sum_{n}\la_f(n)\chi(n)n^{iK\nu}V\left(\frac{n}{N}\right)\sum_r\frac{\chi(r)}{\sqrt{r}}\left(\frac{p(t+K\nu)}{2\pi er}\right)^{-i(t+K\nu)}V_\epsilon\left(\frac{rN}{Pt}\right)W_{1,p,a}(r,\nu)d\nu\nonumber\\
    &+O\left(\frac{\sqrt{PN}}{t^{1-\epsilon}}\right).
\end{align}

\subsection{\texorpdfstring{$n$}{n}-sum functional equation}

Again by Mellin inversion, we have for any $\sigma>1$,
\begin{align*}
    \sum_n\la_f(n)\chi(n)n^{iK\nu}V\left(\frac{n}{N}\right)&=\frac{1}{2\pi i}\int_{(\sigma)}\Tilde{V}(s)N^s\sum_n\la_f(n)\chi(n)n^{-s+iK\nu}ds\\
    &=\frac{1}{2\pi i}\int_{(\sigma)}\Tilde{V}(s)N^sL(s-iK\nu, f\otimes\chi)ds.
\end{align*}
Since $p\sim P> t^\epsilon$ and $M$ is fixed, we have $(M,p)=1$. Let \begin{align*}
    G_k(s)=\frac{\Gamma\left(1-s+\frac{k-1}{2}\right)}{\Gamma\left(s+\frac{k-1}{2}\right)}.
\end{align*}
Applying the functional equation for the twisted $L$-function in Lemma \ref{FElemma} and shifting the contour to $\sigma<0$, the $n$-sum becomes
\begin{align*}
    &\frac{\epsilon(f)\psi(p)\chi(M)\epsilon_\chi^2}{2\pi i}\int_{(\sigma)}\Tilde{V}(s)N^s\left(\frac{p}{2\pi}\right)^{1-2s+2iK\nu}G_k(s-iK\nu)L(1-s+iK\nu,\overline{f}\otimes\overline{\chi})ds\\
    =&\frac{\epsilon(f)\psi(p)\chi(M)\epsilon_\chi^2}{2\pi i}\left(\frac{p}{2\pi}\right)^{1+2iK\nu}\sum_n\overline{\la_f(n)}\overline{\chi}(n)n^{-1-iK\nu}\int_{(\sigma)}\Tilde{V}(s)N^s\left(\frac{p}{2\pi}\right)^{-2s}n^sG_k(s-iK\nu)ds.
\end{align*}

Applying Lemma \ref{MellinBound} on $\Tilde{V}$ together with $\left|G_k(\sigma+i\tau-iK\nu)\right| \ll (1+|K\nu+\tau|)^{-2\sigma}$ from Lemma \ref{GammaRatioBound}, one saves an arbitrary amount unless $|\im(s)|\ll t^\epsilon$. So by shifting the contour to $-M$ for some $M>0$, one saves an arbitrary amount unless $n\ll \frac{P^2K^2t^\epsilon}{N}$. Similarly, by shifting the contour to the right without hitting any poles since $|\im(s)|\ll t^\epsilon$, one saves an arbitrary amount unless $n\gg\frac{P^2K^2}{Nt^\epsilon}$. Hence we can insert the function $V_\epsilon\left(\frac{nN}{P^2K^2}\right)$ introduced in the previous section with the cost of an arbitrarily small error. Shifting the contour back to $\frac{1}{2}$ and truncating $|\im(s)|\ll t^\epsilon$ by $\Tilde{V}$ with the cost of an arbitrarily small error, the $n$-sum is equal to \begin{align}
    &\frac{\epsilon(f)\psi(p)\chi(M)\epsilon_\chi^2}{2\pi i}\left(\frac{p}{2\pi}\right)^{1+2iK\nu}\sum_n\overline{\la_f(n)}\overline{\chi}(n)n^{-1-iK\nu}V_\epsilon\left(\frac{nN}{P^2K^2}\right)\nonumber\\
    &\times\int_{\substack{\left(\frac{1}{2}\right)\\|\im(s)|\ll t^\epsilon}}\Tilde{V}(s)N^s\left(\frac{p}{2\pi}\right)^{-2s}n^sG_k(s-iK\nu)ds+O\left(t^{-2020}\right).
\end{align}

Now by lemma \ref{GammaRatioLemma}, we get \begin{align*}
    G_k(s-iK\nu)&=\frac{\Gamma\left(\frac{k}{2}-i(-K\nu+\tau)\right)}{\Gamma\left(\frac{k}{2}+i(-K\nu+\tau)\right)}\\
    &=\left(\frac{K\nu}{e}\right)^{2i K\nu}(K\nu)^{-2i\tau}e\left(\frac{\tau^2}{2\pi K\nu}\right)W_{\frac{k}{2}}\left(-K\nu\right)+O\left(K^{-2}\right),
\end{align*}
with $W_{\frac{k}{2}}(-K\nu)$ as defined in Lemma \ref{GammaRatioLemma}. Hence the $s$-integral becomes\begin{align*}
    &\int_{\substack{\left(\frac{1}{2}\right)\\|\im(s)|\ll t^\epsilon}}\Tilde{V}(s)N^s\left(\frac{p}{2\pi}\right)^{-2s}n^sG_k(s-iK\nu)ds\\
    =&\left(\frac{K\nu}{e}\right)^{2i K\nu}W_{\frac{k}{2}}\left(-K\nu\right)\int_{\substack{\left(\frac{1}{2}\right)\\|\im(s)|\ll t^\epsilon}}\Tilde{V}(s)N^s\left(\frac{p}{2\pi}\right)^{-2s}n^s(K\nu)^{-2i\tau}e\left(\frac{\tau^2}{2\pi k\nu}\right)\left(1+O\left(K^{-2}\right)\right)ds\\
    =&\frac{2\pi i\sqrt{Nn}}{p}\left(\frac{K\nu}{e}\right)^{2i K\nu}W_{\frac{k}{2}}\left(-K\nu\right)\int_{|\tau|\ll t^\epsilon}\Tilde{V}\left(\frac{1}{2}+i\tau\right)\left(\frac{4\pi^2nN}{p^2K^2\nu^2}\right)^{i\tau}e\left(\frac{\tau^2}{2\pi k\nu}\right)d\tau+O\left(\frac{\sqrt{N}}{Kt^{\frac{1}{2}+\epsilon}}\right).
\end{align*}
Bounding the error term with \eqref{DeligneBound} or \eqref{RPavg}, the $n$-sum becomes 
\begin{align*}
    &\epsilon(f)\psi(p)\chi(M)\epsilon_\chi^2\sqrt{N}\sum_n\frac{\la_f(n)\overline{\chi}(n)}{\sqrt{n}}\left(\frac{p^2K^2\nu^2}{4\pi^2e^2n}\right)^{iK\nu}V_\epsilon\left(\frac{nN}{P^2K^2}\right)W_{2,p}(n,\nu)+O\left(\frac{P\sqrt{N}}{Kt^{\frac{1}{2}+\epsilon}}\right),
\end{align*}
where \begin{align*}
    W_{2,p}(n,\nu)=\frac{1}{2\pi}W_{\frac{k}{2}}\left(-K\nu\right)\int_{|\tau|< t^\epsilon}\Tilde{V}\left(\frac{1}{2}+i\tau\right)\left(\frac{4\pi^2nN}{p^2K^2\nu^2}\right)^{i\tau}e\left(\frac{\tau^2}{2\pi k\nu}\right)d\tau
\end{align*}
is $t^\epsilon$-inert in $n,\nu$.

Putting this back into $S^*(N)$ in (\ref{S*afterrdual}) and bounding the error term with absolute value, we get \begin{align}\label{S*afterBothDual}
    S^*(N)=&\frac{N\epsilon(f)}{P^*}\sum_{p\in \mathcal{P}}\frac{\psi(p)}{\phi(p)}\sumast_{\chi(p)}\chi(M)\epsilon_\chi^2\epsilon_{\overline{\chi}}\int_0^\infty V\left(\nu\right)\sum_r\frac{\chi(r)}{\sqrt{r}}\left(\frac{p(t+K\nu)}{2\pi er}\right)^{-i(t+K\nu)}V_\epsilon\left(\frac{rN}{Pt}\right)W_{1,p,a}(r,\nu)\nonumber\\
    &\times \sum_n\frac{\overline{\la_f(n)}\overline{\chi}(n)}{\sqrt{n}}\left(\frac{p^2K^2\nu^2}{4\pi^2e^2n}\right)^{iK\nu}V_\epsilon\left(\frac{nN}{P^2K^2}\right)W_{2,p}(n,\nu)d\nu+O\left(\frac{P^\frac{3}{2}\sqrt{N}}{K}+\frac{\sqrt{PN}}{t^{1-\epsilon}}\right)\nonumber\\
    =&\frac{N\epsilon(f)(2\pi e)^{it}}{P^*}\sum_{p\in \mathcal{P}}\frac{\psi(p)}{p^{it}\phi(p)}\sumast_{\chi(p)}\epsilon_\chi\chi(-M)\sum_n\frac{\overline{\la_f(n)}\overline{\chi}(n)}{\sqrt{n}}V_\epsilon\left(\frac{nN}{P^2K^2}\right)\sum_r\chi(r)r^{-\frac{1}{2}+it}V_\epsilon\left(\frac{rN}{Pt}\right)\nonumber\\
    &\times \int_0^\infty V\left(\nu\right)W_{1,p,a}(r,\nu)W_{2,p}(n,\nu)\left(\frac{pK^2\nu^2r}{2\pi en(t+K\nu)}\right)^{iK\nu}(t+K\nu)^{-it}d\nu+O\left(\frac{P^\frac{3}{2}\sqrt{N}}{K}+\frac{\sqrt{PN}}{t^{1-\epsilon}}\right).
\end{align}

\begin{remark}
    For the Maass form case with Laplace eigenvalue $\frac{1}{4}+r^2$, one has to replace $G_k(s)$ by \begin{align*}
        \frac{\Gamma\left(\frac{1-s+\delta_f+a(1-2\delta_f)+ir}{2}\right)\Gamma\left(\frac{1-s+\delta_f+a(1-2\delta_f)-ir}{2}\right)}{\Gamma\left(\frac{s+\delta_f+a(1-2\delta_f)+ir}{2}\right)\Gamma\left(\frac{s+\delta_f+a(1-2\delta_f)-ir}{2}\right)},
    \end{align*}
    where \begin{align*}
        a=\begin{cases}0&\text{ if }\chi(-1)=1\\ 1&\text{ if }\chi(-1)=-1,\end{cases} \quad \text{ and } \quad \delta_f=\begin{cases}0&\text{ if } f \text{ is an even Maass form} \\
        1&\text{ if } f \text{ is an odd Maass form}\end{cases}
    \end{align*}
    as defined before Lemma \ref{FElemma}. While this slightly changes the sum we get, the analysis and steps are the same for the rest of the paper to yield the same bound.
\end{remark}

\subsection{Character sum}

Now we evaluate the $\chi$-sum. Recall that the definition of $a$ depends on $\chi(-1)=\pm1$, and hence we have to split the sum into odd and even character sums. First consider the sum
\begin{align}
    \mathcal{C}_\pm:=&\frac{1}{\phi(p)}\sum_{\chi(p)}\frac{\chi(-1)\pm 1}{2}\epsilon_\chi\overline{\chi}(n)\chi(Mr)\nonumber\\
    =&\frac{1}{2\phi(p)}\sum_{\chi(p)}\epsilon_\chi\left(\overline{\chi}(n)\chi(-Mr)\pm\overline{\chi}(n)\chi(Mr)\right).
\end{align}
Writing $\displaystyle\epsilon_\chi=\frac{1}{\sqrt{p}}\sum_{\alpha(p)}e\left(\frac{\alpha}{p}\right)\chi(\alpha)$, we get
\begin{align*}
    \mathcal{C}_\pm&=\frac{1}{2\phi(p)\sqrt{p}}\sum_{\alpha(p)}e\left(\frac{\alpha}{p}\right)\sum_{\chi(p)}\left(\overline{\chi}(n)\chi(-M\alpha r)\pm\overline{\chi}(n)\chi(M\alpha r)\right)\\
    &=\frac{1}{2\sqrt{p}}\delta(p\nmid n)\sum_{\alpha(p)}e\left(\frac{\alpha}{p}\right)\left(\delta(n\equiv-M\alpha r\Mod{p})\pm\delta(n\equiv M\alpha r\Mod{p})\right)\\
    &=\frac{1}{2\sqrt{p}}\delta(p\nmid n,r)\left(e\left(-\frac{n\overline{Mr}}{p}\right)\pm e\left(\frac{n\overline{Mr}}{p}\right)\right).
\end{align*}

Putting this back into $S^*(N)$ in (\ref{S*afterBothDual}), by adding and subtracting the trivial character we get \begin{align}\label{S*Split}
    S^*(N)=S_0^*(N)+S_1^*(N)-T(N)+O\left(\frac{P^\frac{3}{2}\sqrt{N}}{K}+\frac{\sqrt{PN}}{t^{1-\epsilon}}\right),
\end{align}
where for $a=0,1$, \begin{align}\label{S*aDef}
    S_a^*(N)=&\frac{N\epsilon(f)(2\pi e)^{it}}{2P^*}\sum_{p\in \mathcal{P}}p^{-\frac{1}{2}-it}\sum_{p\nmid n}\frac{\overline{\la_f(n)}}{\sqrt{n}}V_\epsilon\left(\frac{nN}{P^2K^2}\right)\sum_{p\nmid r}r^{-\frac{1}{2}+it}V_\epsilon\left(\frac{rN}{Pt}\right)\left(e\left(-\frac{n\overline{Mr}}{p}\right)+(-1)^a e\left(\frac{n\overline{Mr}}{p}\right)\right)\nonumber\\
    &\times \int_0^\infty V\left(\nu\right)W_{1,p,a}(r,\nu)W_{2,p}(n,\nu)\left(\frac{pK^2\nu^2r}{2\pi en(t+K\nu)}\right)^{iK\nu}(t+K\nu)^{-it}d\nu,
\end{align}
and \begin{align}
    T(N)=&\frac{N\epsilon(f)(2\pi e)^{it}}{P^*}\sum_{p\in \mathcal{P}}\frac{1}{p^{it}\phi(p)}\sum_n\frac{\overline{\la_f(n)}}{\sqrt{n}}V_\epsilon\left(\frac{nN}{P^2K^2}\right)\sum_r r^{-\frac{1}{2}+it}V_\epsilon\left(\frac{rN}{Pt}\right)\nonumber\\
    &\times \int_0^\infty V\left(\nu\right)W_{1,p,0}(r,\nu)W_{2,p}(n,\nu)\left(\frac{pK^2\nu^2r}{2\pi en(t+K\nu)}\right)^{iK\nu}(t+K\nu)^{-it}d\nu.
\end{align}

\subsection{Treatment of \texorpdfstring{$T(N)$}{T(N)}}

Consider the $n$-sum in $T(N)$, \begin{align*}
    \sum_n\overline{\la_f(n)}n^{-\frac{1}{2}-iK\nu}V_\epsilon\left(\frac{nN}{P^2K^2}\right)W_{2,p}(n,\nu).
\end{align*}
By Mellin inversion and the $L$-function representation, we have for any $\sigma>1$, \begin{align}\label{S1nsum}
    \sum_{n}\overline{\la_f(n)}n^{-\frac{1}{2}-iK\nu}V\left(\frac{nN}{P^2K^2}\right)=&\frac{1}{2\pi i}\int_{(\sigma)}\Tilde{V}(s)\left(\frac{P^2K^2}{N}\right)^s\sum_n\overline{\la_f(n)}n^{-\frac{1}{2}-iK\nu-s}ds\nonumber\\
    =&\frac{1}{2\pi i}\int_{(\sigma)}\Tilde{V}(s)\left(\frac{P^2K^2}{N}\right)^sL\left(\overline{f},s+\frac{1}{2}+iK\nu\right)ds.
\end{align}
Applying Lemma \ref{MellinBound} on $\Tilde{V}$ gives an arbitrary savings unless $|\im(s)|\ll t^\epsilon$, and hence we can shift the contour to $-M$ for some $M>0$ and apply the functional equation to see that (\ref{S1nsum}) is equal to \begin{align}
    &\frac{\overline{\epsilon(f)}}{2\pi i}\int_{(-M)}\Tilde{V}(s)(2\pi)^{2s+2iK\nu}\left(\frac{P^2K^2}{N}\right)^s\frac{\Gamma\left(\frac{1}{2}-s-iK\nu+\frac{k-1}{2}\right)}{\Gamma\left(s+\frac{1}{2}+iK\nu+\frac{k-1}{2}\right)}L\left(f,\frac{1}{2}-s-iK\nu\right)ds\nonumber\\
    =&\frac{\overline{\epsilon(f)}}{2\pi i}\int_{(-M)}\Tilde{V}(s)(2\pi)^{2s+2iK\nu}\left(\frac{P^2K^2}{N}\right)^s\frac{\Gamma\left(\frac{1}{2}-s-iK\nu+\frac{k-1}{2}\right)}{\Gamma\left(s+\frac{1}{2}+iK\nu+\frac{k-1}{2}\right)}\sum_n\la_f(n)n^{s-iK\nu-1}ds.
\end{align}
We have $\left|\frac{\Gamma\left(\frac{1}{2}-s+iK\nu+\frac{k-1}{2}\right)}{\Gamma\left(s+\frac{1}{2}+iK\nu+\frac{k-1}{2}\right)}\right| \ll 1+|K\nu+\tau|^{2\re(s)}$, and repeated integration by parts on $\Tilde{V}$ gives us an arbitrary savings unless $|\tau|\ll t^\epsilon$. Therefore, one saves an arbitrary amount by taking $M$ large enough unless \begin{align*}
    n\ll\frac{Nt^\epsilon}{P^2}.
\end{align*}
Hence by choosing $P^2>Nt^\epsilon$, one saves an arbitrary amount, that is, \begin{align}\label{TBound}
    T(N)=O\left(t^{-2020}\right).
\end{align}

\begin{remark}
    One can directly apply Voronoi summation to get the same result here. But we intentionally avoid doing so for the purpose of this paper.
\end{remark}

\begin{remark}
    For the Maass form case, one has to adjust the Gamma function in a similar fashion mentioned in the end of Section 6.2, and then the same analysis goes through.
\end{remark}

\subsection{Analysing the \texorpdfstring{$\nu$}{nu}-integral}

Now we analyse the $\nu$-integral given by
\begin{align}\label{IDef}
    I(p,n,r):=&\int_0^\infty V\left(\nu\right)W_{1,p,a}(r,\nu)W_{2,p}(n,\nu)\left(\frac{pK^2\nu^2r}{2\pi en(t+K\nu)}\right)^{iK\nu}(t+K\nu)^{-it}d\nu=\int_0^\infty g(\nu)e(h(\nu))d\nu
\end{align}
where \begin{align*}
    g(\nu)&=V\left(\nu\right)W_{1,p,a}(r,\nu)W_{2,p}(n,\nu)\\
    2\pi h(t)&=2K\nu\log(\nu)-K\nu\log(t+K\nu)+K\nu\log\left(\frac{pK^2r}{2\pi en}\right)-t\log(t+K\nu).
\end{align*}
Now \begin{align*}
    2\pi h'(\nu)&=2K\log(\nu)+2K-K\log(t+K\nu)-\frac{K^2\nu}{t+K\nu}+K\log\left(\frac{pK^2r}{2\pi en}\right)-\frac{Kt}{t+K\nu}\\
    &=2K\log(\nu)-K\log(t+K\nu)+K\log\left(\frac{pK^2r}{2\pi n}\right)\\
    2\pi h''(\nu)&=\frac{K(2t+K\nu)}{\nu(t+K\nu)}\\
    h^{(j)}(\nu)&\ll_j K
\end{align*}
for $j\geq 2$. Choosing $K<t^{1-\varepsilon}$, we have $h''(\nu)\sim K$. Solving $h'(v_0)=0$, we get that the stationary phase is \begin{align*}
    \nu_0=&\frac{\pi n}{pKr}\left(\sqrt{\frac{2prt}{\pi n}+1}+1\right)\\
    h(\nu_0)=&-K\nu_0-t\log(t+K\nu_0).
\end{align*}
By Lemma \ref{stat}, there exists some $t^\varepsilon$-inert function $F$ such that, \begin{align*}
    I(p,n,r)=&V(\nu_0)W_{1,p,a}(r,\nu_0)W_{2,p}(n,\nu_0)\frac{F(\nu_0)}{\sqrt{h''(\nu_0)}}t^{-it}\nonumber\\
    &\times e\left(-\frac{n}{2pr}\left(\sqrt{\frac{2prt}{\pi n}+1}+1\right)-\frac{t}{2\pi}\log\left(1+\frac{\pi n}{prt}\left(\sqrt{\frac{2prt}{\pi n}+1}+1\right)\right)\right)+O(t^{-2020}).
\end{align*}
For the ease of later computation, we rewrite \begin{align*}
    e\left(-\frac{n}{2pr}\left(\sqrt{\frac{2prt}{\pi n}+1}+1\right)\right)=&e\left(-\sqrt{\frac{nt}{2\pi pr}}\sqrt{1+\frac{\pi n}{2prt}}-\frac{n}{2pr}\right)\\
    =&e\left(-\sqrt{\frac{nt}{2\pi pr}}-\frac{n}{2pr}\right)U_{1,p}(n,r),
\end{align*}
where \begin{align*}
    U_{1,p}(n,r)=e\left(-\sqrt{\frac{nt}{2\pi pr}}\left(\sqrt{1+\frac{\pi n}{2prt}}-1\right)\right)
\end{align*}
is $\frac{K^3}{t^{2-\epsilon}}$-inert, and by Taylor expansions, \begin{align*}
    e\left(-\frac{t}{2\pi}\log\left(1+\frac{\pi n}{prt}\left(\sqrt{\frac{2prt}{\pi n}+1}+1\right)\right)\right)=&e\left(-\frac{t}{2\pi}\left(\sqrt{\frac{2\pi n}{prt}}\sqrt{1+\frac{\pi n}{2prt}}+\frac{\pi n}{prt}-\frac{1}{2}\left(\sqrt{\frac{2\pi n}{prt}}\sqrt{1+\frac{\pi n}{2prt}}+\frac{\pi n}{prt}\right)^2\right.\right.\\
    &\left.\left.+\sum_{j=3}^\infty\frac{(-1)^{j-1}}{j}\left(\frac{\pi n}{prt}\left(\sqrt{\frac{2prt}{\pi n}+1}+1\right)\right)^j\right)\right)\\
    =&e\left(-\sqrt{\frac{nt}{2\pi pr}}\right)U_{2,p}(n,r),
\end{align*}
where \begin{align*}
    U_{2,p}(n,r)=e\left(-\frac{t}{2\pi}\left(\sum_{j=2}^\infty\frac{(-1)^{j-1}}{j}\left(\frac{\pi n}{prt}\left(\sqrt{\frac{2prt}{\pi n}+1}+1\right)\right)^j+\frac{\pi n}{prt}\right)\right)
\end{align*}
is $\frac{K^3}{t^{2-\epsilon}}$-inert. Putting this back into $I(p,n,r)$ in (\ref{IDef}), we have \begin{align}
    I(p,n,r)=&V(\nu_0)W_{1,p,a}(r,\nu_0)W_{2,p}(n,\nu_0)\frac{F(\nu_0)}{\sqrt{h''(\nu_0)}}t^{-it}\nonumber\\
    &\times e\left(-\sqrt{\frac{2nt}{\pi pr}}-\frac{n}{2pr}\right)U_{1,p}(n,r)U_{2,p}(n,r)+O(t^{-2020}).
\end{align}

Putting this back into $S^*_a(N)$ in (\ref{S*aDef}), we get \begin{align}
    S_a^*(N)=&\frac{N^2\epsilon(f)(2\pi e)^{it}}{2P^2P^*K^\frac{3}{2}t^{\frac{1}{2}+it}}\sum_{p\in \mathcal{P}}\frac{\psi(p)}{p^{it}}\sum_{p\nmid n}\overline{\la_f(n)}V_\epsilon\left(\frac{nN}{P^2K^2}\right)\sum_{p\nmid r}r^{it}V_\epsilon\left(\frac{rN}{Pt}\right)\left(e\left(-\frac{n\overline{Mr}}{p}\right)+(-1)^a e\left(\frac{n\overline{Mr}}{p}\right)\right)\nonumber\\
    &\times e\left(-\sqrt{\frac{2nt}{\pi pr}}-\frac{n}{2pr}\right)W_p(n,r)+O(t^{-2020}), 
\end{align}
where \begin{align}
    W_p(n,r)=\frac{P^2K^\frac{3}{2}\sqrt{t}}{N\sqrt{pnr}}V(\nu_0)W_{1,p,a}(r,\nu_0)W_{2,p}(n,\nu_0)U_{1,p}(n,r)U_{2,p}(n,r)\frac{F(\nu_0)}{\sqrt{f''(\nu_0)}}.
\end{align}
Differentiating, we see that $W_p(n,r)$ is bounded by $t^\epsilon$ and $\left(1+\frac{K^3}{t^2}\right)t^\epsilon$-inert in $n,r$.

At this point, bounding everything with absolute value yields \begin{align*}
    S^*_a(N)\ll\frac{N^2t^\epsilon}{P^3K^\frac{3}{2}\sqrt{t}}P\frac{P^2K^2}{N}\frac{Pt}{N}=P\sqrt{K}t^{\frac{1}{2}+\epsilon},
\end{align*}
which is not good enough yet as we have the constraint $PK>Nt^\epsilon$.

\section{The Cauchy-Schwarz inequality}\label{SectCS}

Applying the Cauchy-Schwarz inequality and taking out the $n$-sum, we get \begin{align}
    S_a^*(N)^2\leq&\sum_\pm\frac{N^3}{P^4Kt^{1-\epsilon}}\sum_nV_\epsilon\left(\frac{nN}{P^2K^2}\right)\left|\sum_{p\in \mathcal{P}}\frac{\psi(p)}{p^{it}}\sum_{p\nmid r}r^{it}V_\epsilon\left(\frac{rN}{Pt}\right)e\left(\pm\frac{n\overline{Mr}}{p}-\sqrt{\frac{2nt}{\pi pr}}-\frac{n}{2pr}\right)W_p(n,r)\right|^2\nonumber\\
    =&\sum_\pm\frac{N^3}{P^4Kt^{1-\epsilon}}\sum_nV_\epsilon\left(\frac{nN}{P^2K^2}\right)\sum_{p_1\in \mathcal{P}}\sum_{p_2\in \mathcal{P}}\psi(p_1)\overline{\psi(p_2)}\left(\frac{p_2}{p_1}\right)^{it}\sum_{p_1\nmid r_1}\sum_{p_2\nmid r_2}\left(\frac{r_1}{r_2}\right)^{it}V_\epsilon\left(\frac{r_1N}{Pt}\right)V_\epsilon\left(\frac{r_2N}{Pt}\right)\nonumber\\
    &\times e\left(\pm\frac{n\overline{Mr_1}}{p_1}\mp\frac{n\overline{Mr_2}}{p_2}-\sqrt{\frac{2nt}{\pi p_1r_1}}-\frac{n}{2p_1r_1}+\sqrt{\frac{2nt}{\pi p_2r_2}}+\frac{n}{2p_2r_2}\right)W_{p_1}(n,r_1)\overline{W_{p_2}(n,r_2)}.
\end{align}
Performing Poisson summation on the $n$-sum, we get \begin{align}
    &\sum_nV_\epsilon\left(\frac{nN}{P^2K^2}\right)e\left(\pm\frac{n\overline{Mr_1}}{p_1}\mp\frac{n\overline{Mr_2}}{p_2}-\sqrt{\frac{2nt}{\pi p_1r_1}}-\frac{n}{2p_1r_1}+\sqrt{\frac{2nt}{\pi p_2r_2}}+\frac{n}{2p_2r_2}\right)W_{p_1}(n,r_1)\overline{W_{p_2}(n,r_2)}\nonumber\\
    =&\sum_{\gamma\Mod{p_1p_2}}e\left(\pm\frac{\gamma(\overline{Mr_1}p_2-\overline{Mr_2}p_1)}{p_1p_2}\right)\sum_n\int_0^\infty V_\epsilon\left(\frac{(\gamma+p_1p_2y)N}{P^2K^2}\right)W_{p_1}(\gamma+p_1p_2y,r_1)\overline{W_{p_2}(\gamma+p_1p_2y,r_2)}\nonumber\\
    &\times e\left(-\sqrt{\frac{2(\gamma+p_1p_2y)t}{\pi p_1r_1}}-\frac{\gamma+p_1p_2y}{2p_1r_1}+\sqrt{\frac{2(\gamma+p_1p_2y)t}{\pi p_2r_2}}+\frac{\gamma+p_1p_2y}{2p_2r_2}-ny\right)dy\nonumber\\
    =&\frac{P^2K^2}{p_1p_2N}\sum_n\sum_{\gamma\Mod{p_1p_2}}e\left(\frac{\gamma(\overline{Mr_1}p_2-\overline{Mr_2}p_1\pm n)}{p_1p_2}\right)\int_0^\infty V_\epsilon\left(w\right)W_{p_1}\left(\frac{P^2K^2w}{N},r_1\right)\overline{W_{p_2}\left(\frac{P^2K^2w}{N},r_2\right)}\nonumber\\
    &\times e\left(-\sqrt{\frac{2P^2K^2tw}{\pi Np_1r_1}}-\frac{P^2K^2w}{2Np_1r_1}+\sqrt{\frac{2P^2K^2tw}{\pi Np_2r_2}}+\frac{P^2K^2w}{2Np_2r_2}-\frac{nP^2K^2w}{p_1p_2N}\right)dw\nonumber\\
    =&\frac{P^2K^2}{N}\sum_n\delta\left(\pm n\equiv \overline{Mr_2}p_1-\overline{Mr_1}p_2\Mod{p_1p_2}\right)J_{p_1,p_2}(n,r_1,r_2)
\end{align}
where \begin{align}
    J_{p_1,p_2}(n,r_1,r_2)=&\int_0^\infty V_\epsilon\left(w\right)W_{p_1}\left(\frac{P^2K^2w}{N},r_1\right)\overline{W_{p_2}\left(\frac{P^2K^2w}{N},r_2\right)}\nonumber\\
    &\times e\left(-\sqrt{\frac{2P^2K^2tw}{\pi Np_1r_1}}-\frac{P^2K^2w}{2Np_1r_1}+\sqrt{\frac{2P^2K^2tw}{\pi Np_2r_2}}+\frac{P^2K^2w}{2Np_2r_2}-\frac{nP^2K^2w}{p_1p_2N}\right)dw.
\end{align}

Now we analyse the integral $J_{p_1,p_2}(n,r_1,r_2)$. We will choose $K<t^{\frac{2}{3}-\epsilon}$, which implies $W_{p_1}, W_{p_2}$ are both $t^\epsilon$-inert functions. By repeated integration by parts, one saves an arbitrary amount unless \begin{align*}
    |n|\ll\frac{Nt^\epsilon}{K}.
\end{align*}
Hence we get \begin{align}\label{S*AfterCS}
    S_a^*(N)^2\leq& \sum_\pm\frac{N^2K}{P^2t^{1-\epsilon}}\sum_{|n|\ll\frac{Nt^\epsilon}{K}}\sum_{p_1\in \mathcal{P}}\sum_{p_2\in \mathcal{P}}\psi(p_1)\overline{\psi(p_2)}\left(\frac{p_2}{p_1}\right)^{it}\sum_{p_1\nmid r_1}\sum_{p_2\nmid r_2}V_\epsilon\left(\frac{r_1N}{Pt}\right)V_\epsilon\left(\frac{r_2N}{Pt}\right)\left(\frac{r_1}{r_2}\right)^{it}\nonumber\\
    &\times \delta\left(\pm n\equiv \overline{Mr_2}p_1-\overline{Mr_1}p_2\Mod{p_1p_2}\right)J_{p_1,p_2}(n,r_1,r_2)+O(t^{-2020}).
\end{align}
Let \begin{align}
    H(w)=-\sqrt{\frac{2P^2K^2tw}{\pi Np_1r_1}}-\frac{P^2K^2w}{2Np_1r_1}+\sqrt{\frac{2P^2K^2tw}{\pi Np_2r_2}}+\frac{P^2K^2w}{2Np_2r_2}-\frac{nP^2K^2w}{p_1p_2N}
\end{align}
be the phase function. Differentiating, we get \begin{align*}
    H'(w)=&-\sqrt{\frac{P^2K^2t}{2\pi Np_1r_1w}}-\frac{P^2K^2}{2Np_1r_1}+\sqrt{\frac{P^2K^2t}{2\pi Np_2r_2w}}+\frac{P^2K^2}{2Np_2r_2}-\frac{nP^2K^2}{p_1p_2N}\\
    H''(w)=&PK\sqrt{\frac{t}{8\pi Np_1p_2r_1r_2w^3}}\left(\frac{p_2r_2-p_1r_1}{\sqrt{p_1r_1}+\sqrt{p_2r_2}}\right)\gg\frac{NK}{P^2t^{1-\epsilon}}|p_2r_2-p_1r_1|.
\end{align*}
Thus for $K<t^{\frac{2}{3}-\epsilon}$, the second derivative bound in Lemma \ref{2ndDBound} gives \begin{align}\label{JBound}
    J_{p_1,p_2}(n,r_1,r_2)\ll\min\left\{1,\sqrt{\frac{P^2t}{NK|p_2r_2-p_1r_1|}}\right\}t^\epsilon.
\end{align}

Let $\mathcal{D}$ be the contribution of the case $p_1=p_2=p$ to the bound of $S_a^*(N)^2$ in (\ref{S*AfterCS}), that is, \begin{align}
    \mathcal{D}:=&\sum_\pm\frac{N^2K}{P^2t^{1-\epsilon}}\sum_{|n|\leq\frac{Nt^\epsilon}{K}}\sum_{p\in \mathcal{P}}\sum_{p\nmid r_1}\sum_{p\nmid r_2}V_\epsilon\left(\frac{r_1N}{Pt}\right)V_\epsilon\left(\frac{r_2N}{Pt}\right)\left(\frac{r_1}{r_2}\right)^{it}\nonumber\\
    &\times \delta\left(\pm n\equiv \overline{Mr_2}p-\overline{Mr_1}p\Mod{p^2}\right)J_{p,p}(n,r_1,r_2).
\end{align}
The congruence condition implies $p|n$ and $n\equiv \overline{Mr_2}-\overline{Mr_1}\Mod{p}$. Since $PK>Nt^\epsilon$, this implies $n=0$. Hence the congruence condition reduces to $r_1\equiv r_2\Mod{p}$. Further splitting $\mathcal{D}$ into whether $r_1=r_2$ or not, we get \begin{align}
    \mathcal{D}=\mathcal{D}_0+\mathcal{D}_1,
\end{align}
where \begin{align}
    \mathcal{D}_0=\frac{N^2K}{P^2t^{1-\epsilon}}\sum_{p\in \mathcal{P}}\sum_{p\nmid r}V_\epsilon\left(\frac{rN}{Pt}\right)^2J_{p,p}(0,r,r)\ll NKt^\epsilon
\end{align}
by (\ref{JBound}), and \begin{align}
    \mathcal{D}_1=\frac{N^2K}{P^2t^{1-\epsilon}}\sum_{p\in \mathcal{P}}\sum_{p\nmid r_1}\sum_{\substack{p\nmid r_2\\ r_1\equiv r_2\Mod{p}}}\left(\frac{r_1}{r_2}\right)^{it}V_\epsilon\left(\frac{r_1N}{Pt}\right)V_\epsilon\left(\frac{r_2N}{Pt}\right)J_{p,p}(0,r_1,r_2).
\end{align}
Now for $r_1\neq r_2$, in order to use (\ref{JBound}) to bound the $J_{p_1,p_2}(n,r_1,r_2)$, we split the difference $\left|\frac{r_2-r_1}{p}\right|$ into dyadic segments to get \begin{align}
    \mathcal{D}_1=&\frac{N^2K}{P^2t^{1-\epsilon}}\sum_{0\leq\eta<t^\epsilon}\sum_{p\in \mathcal{P}}\sum_{p\nmid r_1}\sum_{\substack{p\nmid r_2\\ r_1\equiv r_2\Mod{p}}}\left(\frac{r_1}{r_2}\right)^{it}V_\epsilon\left(\frac{r_1N}{Pt}\right)V_\epsilon\left(\frac{r_2N}{Pt}\right)J_{p,p}(0,r_1,r_2)\delta\left(p2^\eta<|r_2-r_1|\leq p2^{\eta+1}\right)\nonumber\\
    \ll&\frac{N^2K}{P^2t^{1-\epsilon}}\sup_{0\leq\eta<t^\epsilon}\sum_{p\in \mathcal{P}}\sum_{p\nmid r_1}\sum_{\substack{p\nmid r_2\neq r_1\\ r_1\equiv r_2\Mod{p}}}V_\epsilon\left(\frac{r_1N}{Pt}\right)V_\epsilon\left(\frac{r_2N}{Pt}\right)\sqrt{\frac{t}{NK2^\eta}}\delta\left(p2^\eta<|r_2-r_1|\leq p2^{\eta+1}\right)\nonumber\\
    \ll&\frac{N^2K}{P^2t^{1-\epsilon}}\sup_{0\leq\eta<t^\epsilon}P\frac{Pt}{N}\min\left\{\frac{t}{N},2^\eta\right\}\sqrt{\frac{t}{NK2^\eta}}\ll \sqrt{K}t^{1+\epsilon}.
\end{align}

Let $\mathcal{O}$ be the contribution of the case $p_1\neq p_2$ to the bound of $S_a^*(N)^2$ in (\ref{S*AfterCS}), that is, \begin{align}
    \mathcal{O}:=&\sum_\pm\frac{N^2K}{P^2t^{1-\epsilon}}\sum_{|n|\leq\frac{Nt^\epsilon}{K}}\sum_{p_1\in \mathcal{P}}\sum_{p_2\neq p_1\in \mathcal{P}}\psi(p_1)\overline{\psi(p_2)}\left(\frac{p_2}{p_1}\right)^{it}\sum_{p_1\nmid r_1}\sum_{p_2\nmid r_2}V_\epsilon\left(\frac{r_1N}{Pt}\right)V_\epsilon\left(\frac{r_2N}{Pt}\right)\left(\frac{r_1}{r_2}\right)^{it}\nonumber\\
    &\times \delta\left(\pm n\equiv \overline{Mr_2}p_1-\overline{Mr_1}p_2\Mod{p_1p_2}\right)J_{p_1,p_2}(n,r_1,r_2)
\end{align}
As $p_1\neq p_2$, we have $(p_1,p_2)=1$. The congruence condition implies that $n\neq 0$, $p_1r_1\neq p_2r_2$, $r_1\equiv \mp\overline{Mn}p_2\Mod{p_1}$ and $r_2\equiv \pm\overline{Mn}p_1\Mod{p_2}$. In order to use (\ref{JBound}) to bound the $J_{p_1,p_2}(n,r_1,r_2)$, we split the difference $|p_2r_2-p_1r_1|$ into dyadic segments, \begin{align}
    \mathcal{O}=&\frac{N^2K}{P^2t^{1-\epsilon}}\sum_\pm\sum_{0\leq\eta<t^\epsilon}\sum_{0<|n|\leq\frac{Nt^\epsilon}{K}}\sum_{p_1\in \mathcal{P}}\sum_{p_2\neq p_1\in \mathcal{P}}\psi(p_1)\overline{\psi(p_2)}\left(\frac{p_2}{p_1}\right)^{it}\sum_{p_1\nmid r_1}\sum_{p_2\nmid r_2}\left(\frac{r_1}{r_2}\right)^{it}V_\epsilon\left(\frac{r_1N}{Pt}\right)V_\epsilon\left(\frac{r_2N}{Pt}\right)\nonumber\\
    &\times\delta\left(r_1\equiv \mp\overline{Mn}p_2\Mod{p_1},r_2\equiv \pm\overline{Mn}p_1\Mod{p_2}, 2^\eta\leq |p_2r_2-p_1r_1|<2^{\eta+1}\right)J_{p_1,p_2}(n,r_1,r_2).
\end{align}
Applying (\ref{JBound}), we get \begin{align}
    \mathcal{O}\ll&\frac{N^2K}{P^2t^{1-\epsilon}}\sum_\pm\sum_{0\leq\eta<t^\epsilon}\sum_{0<|n|\leq\frac{Nt^\epsilon}{K}}\sum_{p_1\in \mathcal{P}}\sum_{p_2\neq p_1\in \mathcal{P}}\sum_{r_1\equiv \mp\overline{Mn}p_2\Mod{p_1}}\sum_{\substack{r_2\equiv \pm\overline{Mn}p_1\Mod{p_2}\\2^\eta\leq |p_2r_2-p_1r_1|<2^{\eta+1}}}V_\epsilon\left(\frac{r_1N}{Pt}\right)V_\epsilon\left(\frac{r_2N}{Pt}\right)\nonumber\\
    &\times \min\left\{1,\sqrt{\frac{P^2t}{NK2^\eta}}\right\}\nonumber\\
    \ll &\frac{N^2K}{P^2t^{1-\epsilon}}\sup_{0\leq\eta<t^\epsilon}\frac{N}{K}P^2\frac{t}{N}\min\left\{\frac{2^\eta}{P^2},\frac{t}{N}\right\}\min\left\{1,\sqrt{\frac{P^2t}{NK2^\eta}}\right\}\ll\frac{Nt^{1+\epsilon}}{\sqrt{K}}.
\end{align}

\section{Final Bound}\label{SectFinal}

Combining the above cases $\mathcal{D}_0,\mathcal{D}_1$ and $\mathcal{O}$, and collecting all the constraints we have set for $P$ and $K$, we get for $PK>Nt^\epsilon$, $P^2>Nt^\epsilon$, $K<t^{\frac{2}{3}+\epsilon}$, $K<Nt^{-\epsilon}$,
\begin{align}
    S_a^{*2}(N)\ll NKt^\epsilon+\sqrt{K}t^{1+\epsilon}+\frac{Nt^{1+\epsilon}}{\sqrt{K}}.
\end{align}
Putting this bound and (\ref{TBound}) into (\ref{S*Split}), we get \begin{align}
    S^*(N)\ll \left(\sqrt{NK}+K^\frac{1}{4}\sqrt{t}+\frac{\sqrt{Nt}}{K^\frac{1}{4}}+\frac{P^\frac{3}{2}\sqrt{N}}{K}+\frac{\sqrt{PN}}{t}\right)t^\epsilon.
\end{align}
Putting this bound into (\ref{SNS*Nrelation}), we get \begin{align}
    S(N)\ll \left(\frac{N}{P}+\frac{N^2}{PK}+\sqrt{NK}+K^\frac{1}{4}\sqrt{t}+\frac{\sqrt{Nt}}{K^\frac{1}{4}}+\frac{P^\frac{3}{2}\sqrt{N}}{K}+\frac{\sqrt{PN}}{t}\right)t^\epsilon.
\end{align}
We will use this bound when $N>t^{\frac{2}{3}+\epsilon}$, and in this case the optimal choices of the parameters under the constraints are \begin{align*}
    P=t^{\frac{1}{2}+\epsilon} \text{ and } K=t^{\frac{2}{3}-\epsilon},
\end{align*}
giving us for $N>t^{\frac{2}{3}+\epsilon}$, \begin{align}
    S(N)\ll \sqrt{N}t^{\frac{1}{3}+\epsilon}.
\end{align}
Combining with the bound (\ref{trivialBound}) for the case $N<t^{\frac{2}{3}+\epsilon}$, we get \begin{align}
    \sup_{N\leq t^{1+\epsilon}}S(N)\ll \min\left\{Nt^\epsilon, \sqrt{N}t^{\frac{1}{3}+\epsilon}\right\}.
\end{align}
This concludes Theorem \ref{thm2}, and hence concludes Theorem \ref{thm1} by (\ref{ApplyAFE}).

\begin{remark}
    The constraints $P^2>Nt^\epsilon$, $K<t^{\frac{2}{3}+\epsilon}$ can be relaxed and the error terms $\frac{N^2t^\epsilon}{PK}, \frac{P^\frac{3}{2}\sqrt{N}}{K}, \frac{\sqrt{PN}}{t}$ can be improved by writing out more terms when doing asymptotic expansions. However, they are not affecting the final bound as the main term $\sqrt{NK}+K^\frac{1}{4}\sqrt{t}+\frac{\sqrt{Nt}}{K^\frac{1}{4}}$ dominates all these terms and constraints.
\end{remark}

\printbibliography

\end{document}